\documentclass[reqno]{amsart}
\usepackage[fontsize=12pt]{scrextend}
\usepackage[utf8]{inputenc}
\usepackage{mathptmx}

\usepackage{amsmath}
\usepackage{amsfonts}
\usepackage{amssymb}
\usepackage{amsthm}
\usepackage{yhmath} 
\usepackage{xfrac}
\usepackage{mathrsfs}
\usepackage{hyperref}

\usepackage[left=1.5in, right=1.5in, top=1.25in, bottom=2in]{geometry}

\newtheorem{defi}{Definition}[section]

\theoremstyle{definition}
\newtheorem{theorem}[defi]{Theorem}
\newtheorem*{theorem*}{Theorem}
\newtheorem{lemma}[defi]{Lemma}
\newtheorem*{coro}{Corollary}

\renewcommand{\epsilon}{\varepsilon}
\newcommand{\R}{\mathbb{R}}
\renewcommand{\P}[1]{\mathbb{P}\left[#1\right]}
\newcommand{\Bord}[2]{\text{Bor}^d(#1,#2)}
\newcommand{\Bor}[1]{\text{Bor}^d(#1)}
\renewcommand{\d}[2]{\lVert#1-#2\rVert}
\newcommand{\diam}[1]{\text{diam}\left(#1\right)}
\newcommand{\B}[2]{\wideparen{B}\left(#1,#2\right)}
\newcommand{\norm}[1]{\left\lVert#1\right\rVert}
\newcommand{\area}[1]{\mathscr{A}\!\!\left(#1\right)}
\newcommand{\vol}[1]{\mathcal{V}\!\left(#1\right)}

\newcommand{\pois}[1]{\text{Pois}\left(#1\right)}
\newcommand{\Borone}[2]{\text{Bor}^{1}(#1,#2)}

\title{The chromatic number of random Borsuk graphs}
\author{Matthew Kahle}
\address{Ohio State University}
\email{kahle.70@osu.edu}
\thanks{MK is grateful for partial support from NSF grant DMS \#1352386.}
\author{Francisco Martinez-Figueroa}
\address{Ohio State University}
\email{martinezfigueroa.2@osu.edu}
\date{\today}

\begin{document}
	
	\maketitle\
	
	\begin{abstract}
We study a model of random graph where vertices are $n$ i.i.d. uniform random points on the unit sphere $S^d$ in $\mathbb{R}^{d+1}$, and a pair of vertices is connected if the Euclidean distance between them is at least $2- \epsilon$. We are interested in the chromatic number of this graph as $n$ tends to infinity.

It is not too hard to see that if $\epsilon > 0$ is small and fixed, then the chromatic number is $d+2$ with high probability. We show that this holds even if $\epsilon \to 0$ slowly enough. We quantify the rate at which $\epsilon$ can tend to zero and still have the same chromatic number. The proof depends on combining topological methods (namely the Lyusternik--Schnirelman--Borsuk theorem) with geometric probability arguments. The rate we obtain is best possible, up to a constant factor --- if $\epsilon \to 0$ faster than this, we show that the graph is $(d+1)$-colorable with high probability.
	\end{abstract}
	
\section{Introduction}

Given $\epsilon>0$ and $d \ge 1$, the Borsuk Graph $\text{Bor}^d(\epsilon)$ is the graph with vertex set corresponding to points on the $d-$dimensional unit sphere $S^d\subset\R^{d+1}$ and edges $\{x,y\}$ if and only if $\d{x}{y}>2-\epsilon$, that is, if the two points are $\epsilon$-near to antipodal. Here distance is measured in the ambient Euclidean space $\R^{d+1}$. It is well known that when $\epsilon$ is sufficiently small, its chromatic number is $d+2$, in fact this is equivalent to the Borsuk--Ulam theorem.

The Borsuk graph was part of Lov\'asz's inspiration for his proof of the Kneser conjecture \cite{Lovasz1978}. Among other properties, this graph constitutes a nice example of a graph with large chromatic number and odd girth. See for example \cite{Gabor2011,Prosanov-Raigorodskii-Sagdeev2017,Erdos-Hajnal1966,Erdos-Hajnal1967}. It has also been studied because of its relation with Borsuk's conjecture and distance graphs \cite{Raigorodskii2012, Barg2014,Prosanov2018,Sagdeev18}.

We are interested in the chromatic number of random induced $n$-vertex subgraphs of the Borsuk graph. Our main point is that if $\epsilon \to 0$ slowly enough as $n \to \infty$, then topological lower bounds on chromatic number are tight. This contrasts with the situation studied by Kahle in \cite{Kahle2007}, where topological lower bounds are not efficient for the chromatic number of Erd\H{o}s--R\'enyi random graphs. Similar problems have also been studied for random Kneser graphs in \cite{Kupavskii2018} and \cite{Kiselev-Kupavskii2018}.

The rest of the paper is organized as follows. We finish this section with some definitions and notation. In section \ref{section_2} we prove Theorem \ref{thm_eps_constant} when $\epsilon$ is fixed. 

\begin{theorem}\label{thm_eps_constant}
	Let $d\geq 1$ and $0<\epsilon<2-\lambda_d$ be fixed. Then a.a.s. $$\chi\left(\Bord{\epsilon}{n}\right)=d+2.$$
\end{theorem}

In section \ref{section_3} we prove Theorem \ref{thm_lower_bound}, stating that the chromatic number is still the same when $\epsilon\to0$ slowly.

\begin{theorem}\label{thm_lower_bound}
	Let $\epsilon(n)=C\left(\dfrac{\log{n}}{n}\right)^{2/d}$, where $$C\geq \dfrac{64}{3} \left( \dfrac{3\pi^2}{4} \right)^{1/d}.$$ Then a.a.s. $\chi\left(\Bord{\epsilon(n)}{n}\right)=d+2.$
\end{theorem}

Finally in section \ref{section_4} we prove Theorem \ref{thm_upper_bound}, showing that this rate is tight, up to a constant, in the sense that if $\epsilon\to 0$ faster, then the random Borsuk graph is $(d+1)$-colorable, a.a.s.

\begin{theorem}\label{thm_upper_bound}
Let $\epsilon(n) = C (\log n / n)^{2/d}$,  where $$C< \frac{3(4-\lambda_d^2)}{64}\sqrt[d]{\frac{9}{4 d^2}}.$$
Then a.a.s. $\chi\left(\Bord{\epsilon(n)}{n}\right)\leq d+1$.
\end{theorem}
	
\begin{defi}[Random Borsuk graph]
	Given $n \ge 1$, $d\geq 1$, and $\epsilon>0$, we define a \textbf{random Borsuk graph}, $\Bord{\epsilon}{n}$, as follows.
		
\begin{itemize}
		\item Its vertices are $X_1, X_2,\cdots,X_n$, $n$ independent and identically distributed uniform random variables over the $d$-dimensional sphere $S^d\subset \R^{d+1}$ of radius 1. 
		\item $X_i$ and $X_j$ for $i\neq j$ are connected by an edge, if and only if $\d{X_i}{X_j}>2-\epsilon$, where $\lVert\cdot\rVert$ is the Euclidean distance.
	\end{itemize}
		
\end{defi}

Throughout this paper we will think of random Borsuk graphs on $S^d$ for a fixed dimension $d$. However we will explicitly point out the constants that depend on $d$ in the statement of the results. We will denote the closed ball with center $x$ and radius $r$ by $B(x,r)=\left\{y\in\R^{d+1}: \norm{x-y}\leq r \right\}$. Similarly, we denote intersections of closed balls with the $d$-sphere by $\B{x}{r}$, and we call them spherical caps, so $$\B{x}{r}:=B(x,r)\cap S^d=\left\{y\in S^{d}: \norm{x-y}\leq r\right\}$$		
Given a Borel set $F\subset \R^{d+1}$, we denote its volume, i.e. Lebesgue measure, as $\vol{F}$. Similarly, for a Borel set $F\subset S^d$, we denote its area on the surface of the sphere, by $\area{F}$. Also, we denote $\omega_d=\vol{S^d}$ and $\alpha_d=\area{S^d}$. Given a graph $G$, we denote its chromatic number by $\chi(G)$. 
				
We say that an event happens \emph{asymptotically almost surely (a.a.s)} if the probability approaches $1$ as $n \to \infty$.
	
\section{Random Borsuk Graph with $\epsilon$ constant}\label{section_2}

We start by proving that when $\epsilon>0$ is constant and small, $\chi(\Bord{\epsilon}{n})=d+2$ a.a.s.

\begin{lemma}\label{lemma_min_distance}
	For $x,y\in S^d$, $\norm{x-y}>2-\epsilon$ if and only if $\norm{x+y}<2\sqrt{\epsilon-\frac{\epsilon^2}{4}}$
\end{lemma}	
\begin{proof}
	Since $\overline{(-x)x}$ is a diameter, $\overline{(-x)y}\perp \overline{xy}$. Thus $\norm{x+y}^2=4-\norm{x-y}^2$, so the claim follows.
\end{proof}

Before getting into the analysis of the chromatic number, let us point out the fact that the odd girth of the Borsuk graph is $> 1/\sqrt{\epsilon}$. While this has been observed before (see \cite{Erdos-Hajnal1966,Erdos-Hajnal1967,Gabor2011}), we include a proof for completeness.

\begin{lemma}\label{lemma_odd_girth}    
    Let $\epsilon>0$ and $x_0\in S^d$. If $x_0y_1x_1y_2\cdots x_ny_{n+1}=x_0$ is an odd cycle in the Borsuk graph $\Bor{\epsilon}$, then $2n+1\geq 1/\sqrt{\epsilon}$. In other words, all odd cycles in $\Bor{\epsilon}$ have length greater than $1/\sqrt{\epsilon}$. 
\end{lemma}

\begin{proof}
    Since $\norm{x_i-y_i},\norm{y_i-x_{i-1}}>2-\epsilon$, for any $i$, by applying Lemma \ref{lemma_min_distance}, we get
    
    \begin{align*}
    \norm{x_i-x_{i-1}} &\leq \norm{x_i+y_i}+\norm{-y_i-x_{i-1}} \\
    & =\norm{x_i+y_i} +\norm{y_i+x_{i-1}} \\
    & \leq 4 \sqrt{\epsilon - {\epsilon^2}/{4}} \\
    & <4\sqrt{\epsilon}.
    \end{align*}
    
    Thus $$\norm{x_n-x_0}\leq\norm{x_n-x_{n-1}}+\norm{x_{n-1}-x_{n-2}}+\cdots+
    \norm{x_1-x_0}\leq 4n\sqrt{\epsilon}.$$
    
    Finally, $$2=\norm{2x_0}=\norm{x_0+y_{n+1}}\leq \norm{x_0-x_n}+\norm{x_n+y_{n+1}}<2(2n+1)\sqrt{\epsilon}.$$
    
    Therefore $$\frac{1}{\sqrt{\epsilon}}<2n+1.$$
    
\end{proof}

\begin{lemma}\label{lemma_color_d+2}
	For each $d\geq 1$, there exist a constant $\lambda_d<2$ such that, whenever $0<r<2-\lambda_d$, the Borsuk graph $\text{Bor}^d(r)$ has a proper coloring with $d+2$ colors. 
\end{lemma}
\begin{proof}
	Let $\Delta$ be the regular $(d+1)$-simplex inscribed in the unit $d$-sphere $S^d$. Consider the map $\Phi:\partial\Delta\to S^d$ from the boundary of $\Delta$ to $S^d$ given by $\Phi(x)=x/\norm{x}$. Note then that $\Phi$ is a homeomorphism. Let $\tau\in\partial\Delta$ be a maximal face, and let $\lambda_d=\diam{\Phi(\tau)}$. Since $\Delta$ is regular, the value of $\lambda_d$ does not depend on the face $\tau$. 
	
	Note now that $\lambda_d<2$. To see this, suppose that $\lambda_d=2$. Since $\tau$ is closed, so is $\Phi(\tau)$, so there exist $x,y\in\tau$ such that $\norm{\Phi(x)-\Phi(y)}=2$. This means $\Phi(x)$ and $\Phi(y)$ are antipodal, and so $y=-\frac{\norm{y}}{\norm{x}}x$. Since $\tau$ is convex, $0=\sfrac{\norm{y}}{(\norm{x}+\norm{y})}x+\sfrac{\norm{y}}{(\norm{x}+\norm{y})}y$ must also be in $\tau$, but this is a contradiction, since $\tau\subset\partial\Delta$, proving the claim.
	
	We now give a coloring for $S^d$ as follows. We start by coloring $\partial\Delta$: give a different color to each of the $d+2$ facets, and for the lower dimensional faces, assign an arbitrary color among the facets that contain them. Finally, color $\Phi(x)\in S^d$, with the color of $x$.
	
	Note this is indeed a proper coloring for $\Bor{r}$, since all points in $S^d$ of the same color lie on the image of a facet $\Phi(\tau)$, of diameter $\lambda_d$; so if $x$ and $y$ have the same color, $\norm{x-y}\leq\lambda_d< 2-r$ so they are not connected by an edge in the Borsuk graph. 
\end{proof}


The upper bound for the chromatic number follows immediately from Lemma \ref{lemma_color_d+2}. The proof we give below for the lower bound, is a direct application of the Lyusternik--Shnirelman--Borsuk Theorem \cite{LS30,Borsuk33}. We state this well-known theorem without proof; for more details and a self-contained proof see, for example, Chapter 2 of Matousek's book \cite{Matousek2008}.

\begin{theorem*}[Lyusternik--Shnirelman--Borsuk] For any cover $U_1, \dots , U_{d+1}$ of the sphere $S^d$ by $d+1$ open (or closed) sets, there is at least one set containing a pair of antipodal points.
\end{theorem*} 

We note that B\'ar\'any gave a short proof of Kneser's conjecture using this theorem \cite{Barany78}. See also Greene's proof \cite{Greene02}. For the rest of the paper, we refer to this theorem as the LSB Theorem.

\begin{proof}[Proof of Theorem \ref{thm_eps_constant}]
	By Lemma \ref{lemma_color_d+2}, since $\Bord{\epsilon}{n}\subset \text{Bor}^d(\epsilon)$, $d+2$ is an upper bound for the chromatic number of the random Borsuk graph.\\ 
	
	Let $F_1,F_2,\dots, F_N$ be a cover of $S^d$ by Borel sets, such that $\text{diam}(F_i)\leq\frac{\sqrt{\epsilon}}{2}$ and $\area{F_i}>0$ for all $i$. Note we can construct such a family of sets in many ways, for instance, as we do in the next section, we can consider a $\delta$-net of $S^d$ and let the sets $F_i$ to be spherical caps centered on the $\delta$-net of radius $\sqrt{\epsilon}/4$ where $\delta\leq\sqrt{\epsilon}/4$. Note here that the sets $F_i$ and $N$ depend only on $\epsilon$, which is fixed. \\
	
	Let $c=\min_{i}\frac{\area{F_i}}{\area{S^d}}$ and $G=\Bord{\epsilon}{n}$. The following computation shows that, a.a.s., $G$ contains at least one vertex in each of the sets $F_i$.
	
	\begin{align*}\label{eqn-prob-Fi-has-vertices}
	\P{\bigwedge_{i=1}^N \left(V(G)\cap F_i \neq \emptyset\right)} & = 
	1-\P{\bigvee_{i=1}^N\left(V(G)\cap F_i=\emptyset\right)}\\
	&\geq 1 - \sum_{i=1}^N\P{V(G)\cap F_i=\emptyset}\nonumber\\
	& = 1 - \sum_{i=1}^N \left( 1-\frac{\area{F_i}}{\area{S^d}}\right)^n  \geq 1 -N(1-c)^n
	\end{align*}
	since $N$ and $c$ are constant, $1-N(1-c)^n\to 1$ as $n\to\infty$, proving the claim. \\
		   
	We may assume then, $G$ has a vertex $y_i\in F_i$ for $i=1,\dots, N$. Proceeding by way of contradiction, suppose there exists a proper coloring of $G$ with $d+1$ colors. For each $j=1,\dots, d+1$ define
	$$U_j=\bigcup \B{y_k}{\frac{\sqrt{\epsilon}}{2}}$$
	where the union is taken over all the $y_k$'s of color $j$.\\
	
	Since $F_i\subset \B{y_i}{\frac{\sqrt{\epsilon}}{2}}$, the sets $U_1,\dots, U_{d+1}$ are a closed cover of $S^d$. Thus, by the LSB Theorem, there exists an antipodal pair in one of the closed sets. Without lost of generality, say $x,(-x)\in U_1$, so $x\in \B{y_1}{\frac{\sqrt{\epsilon}}{2}}$ and $(-x)\in \B{y_2}{\frac{\sqrt{\epsilon}}{2}}$, with both $y_1$ and $y_2$ having color 1. Then,	
	$$\norm{y_1+y_2}\leq \norm{y_1-x}+\norm{x+y_2}\leq \sqrt{\epsilon}\leq 2\sqrt{\epsilon-\frac{\epsilon^2}{4}},$$
	where the last inequality holds because $\epsilon<2$. Lemma \ref{lemma_min_distance} then implies $\norm{y_1-y_2}>2-\epsilon$, so $y_1$ and $y_2$ are connected by an edge in $G$, giving the desired contradiction. 
\end{proof}

\section{Random Borsuk Graph with $\epsilon\to 0$}\label{section_3}
\subsection{Lower Bound}
The proof we gave for the lower bound in Theorem \ref{thm_eps_constant}, suggests that we should be able to let $\epsilon\to 0$ and still a.a.s.\ get the same chromatic number. Indeed, this will be the case. We just need to control the number of sets $N$ we use to cover the sphere and their area, in such a way that 
$$\lim_{n \to \infty} 1 - N (1-c)^n \to 1.$$
In this section we discuss how to do this using $\delta$-nets on $S^d$, and then adapt the proof of Theorem \ref{thm_eps_constant} to get Theorem \ref{thm_lower_bound}.

We start with a technical lemma on spherical caps.

\begin{lemma}\label{lemma_area_spherical_cap}
	Given $x\in S^d$ and $0<r<1$, the following hold for the spherical cap $B=\B{x}{r}$.
	\begin{enumerate}
		\item The boundary $\partial B$, is a $(d-1)$-dimensional sphere with radius $$r'=r\sqrt{1-r^2/4}.$$
		\item $B$ is indeed a cap, i.e. there exist a $d$-hyperplane in $\R^{d+1}$, such that $B$ is the portion of $S^d$ contained in one of the semi-spaces defined by the hyperplane. 
		\item The area $\area{B}$ satisfies the inequalities
		
		$$\frac{1}{\pi}\left(\frac{\sqrt{3}}{2}\right)^{d-1}r^d\leq \frac{\area{\B{x}{r}}}{\area{S^d}}\leq\frac{d}{3} r^d. $$
		
	\end{enumerate}
	
\end{lemma}

\medskip

Spherical caps are well studied in the literature. See, for example, Lemmas 2.2. and 2.3 in \cite{Ball97}. We note that Lemma 2.2 in \cite{Ball97} is better than our Lemma \ref{lemma_area_spherical_cap} in the case that $r$ is fixed and $d \to \infty$, but we are interested in the case that $d$ is fixed. 

\begin{proof}[Proof of Lemma \ref{lemma_area_spherical_cap}]
Without loss of generality, we may choose $x=N=(0,\dots,0,1)$ to be the north pole by rotating the sphere. Thus $x\in\partial B$ if and only if $\norm{x}=1$ and $\norm{x-N}=r$. Then
	\begin{align*}
	r^2 &=\norm{x-N}^2 \\
	&=x_0^2+\cdots +x_{d-1}^2+(x_d-1)^2 \\
	&=x_0^2+\cdots+x_{d-1}^2+x_d^2-2x_d+1 \\
	&=2-2x_d. \\
	\end{align*}

So $x_d=1-\frac{r^2}{2}$. Therefore, the hyperplane $\{x_d=1-\frac{r^2}{2}\}$ determines the cap $B$, proving (2). To get (1), note $1=\norm{x}^2=x_0^2+\cdots+x_d^2$, so we get $x_0^2+\cdots+x_{d-1}^2=r^2-r^4/4=(r')^2$, for all $x\in\partial B$, proving it is indeed a $(d-1)$-dimensional sphere with the desired radius.\\
	
For (3), recall we get the area of $B=\B{x}{r}$ by integrating the length $\ell$ of the arc from $x$ to the boundary $\partial B$, over all the possible unit vectors. Thus
	$$\area{\B{x}{r}}=\int_{\hat{u}\in S^{d-1}} \ell (r')^{d-1}d\!\hat{u}=\ell(r')^{d-1}\int_{\hat{u}\in S^{d-1}}d\!\hat{u}=\ell(r')^{d-1}\alpha_{d-1}.$$
	
	Some planar geometry gives the bounds $$r\leq \ell=2\arcsin(r/2)\leq\frac{\pi}{3}r,$$ and $$\frac{\sqrt{3}}{2}r\leq r'\leq r,$$ since $0<r<1$. This gives
	$$\left(\frac{\sqrt{3}}{2}\right)^{d-1}r^d\alpha_{d-1}\leq \area{B}\leq \frac{\pi}{3}r^d\alpha_{d-1}$$
	
	Recall the formula for the surface area of the unit $d$-dimensional sphere $$\alpha_d=\dfrac{2\pi^{(d+1)/2}}{\Gamma\left((d+1)/2\right)}.$$
By using the fact that the Gamma function is increasing on $[2,\infty)$ and treating the first few cases separately, we have that for $d\geq 1$,	
	$$\frac{1}{\pi}\leq\frac{\alpha_{d-1}}{\alpha_d}=\frac{\Gamma((d+1)/2)}{\sqrt{\pi}\ \Gamma(d/2)}\leq\frac{d}{\pi},$$
from which the desired result follows. 	

\end{proof}

For the sake of completeness, we include the following discussion on $\delta$-nets. See, for example, \cite[Chapter 13]{Matousek2002} for more details. 

\begin{defi}[$\delta$-Nets]
	Given a metric space $X$ with metric $d$, a \textbf{$\boldsymbol{\delta}$-net} is a subset $\mathcal{B}\subset X$ such that for every $x\in X$, there exists $y\in \mathcal{B}$ with $d(x,y)<\delta$.
\end{defi}

We can then construct a $\delta$-net for any compact metric space $M$ inductively. Indeed, choose any point $y_1\in M$.
For each $m\geq 2$, if $B_m=\cup_{i=1}^m \B{y_i}{\delta}\subsetneq M$, choose any $y_{m+1} \in  M \setminus B_m$. Otherwise, stop and let $\mathcal{B}=\{y_1,y_2,\dots\,y_m \}.$

Compactness ensures that the process stops. It's clear that $\mathcal{B}$ is a $\delta$-net and moreover it is also a maximal $\delta$-apart set. That is, $d(y_i,y_j)>\delta$ whenever $i\neq j$, and we can not add any other point to $\mathcal{B}$ without destroying this property. This implies that the balls of radius $\delta$ and center on the points $y_i$'s cover $M$, while the open balls of radius $\delta/2$ with center on the $y_i$'s are all disjoint. We now show that we can control the size of the $\delta$-net in the case that $M=S^d$. 

\begin{lemma}\label{lemma-nets}
	For every $d \ge 1$ and $0<\delta<1$ there exists a $\delta$-net $\mathcal{B}\subset S^d$ , such that (1) for every two points $y_i, y_j\in\mathcal{B}$, $\norm{y_i-y_j}>\delta$, and (2) its cardinality $\mathcal{N}=|\mathcal{B}|$ satisfies:
	$$\frac{3}{d\delta^d}\leq \mathcal{N}\leq \frac{2(3^d)(d+1)}{\delta^d}.$$ 
\end{lemma}

\medskip

In the literature, it seems more common to find an upper bound such as
$$\mathcal{N}\leq(4 / \delta)^{d+1},$$ which is a better bound when $\delta$ is constant and $d\to\infty$ \cite[Lemma 13.1.1]{Matousek2002}. However, the bound we give is more useful for us since we are dealing with $d$ constant and $\delta\to0$.

\begin{proof}[Proof of Lemma \ref{lemma-nets}]
	By letting $\mathcal{B}$ be the $\delta$-net defined above, we already have a $\delta$-net for the sphere $S^d$ that is also a $\delta$-apart maximal set. To prove the inequalities on its cardinality we give a volume and an area arguments. \\
	
	Since the points $y_i\in\mathcal{B}$ are $\delta$ apart, the open balls $\text{int}\left(B\left(y_i,\frac{\delta}{2}\right)\right)$ are disjoint. So
	$$ \vol{\bigcup_{i=1}^N B\left(y_i,\frac{\delta}{2}\right)}= \mathcal{N}\omega_{d}\left(\frac{\delta}{2}\right)^{d+1},$$
where $\omega_{d}$ denotes the volume of the $d+1$-dimensional unit ball in $\R^{d+1}$.	Morever, all such balls are contained in the set $B\left(0,1+\frac{\delta}{2}\right) \setminus B\left(0,1-\frac{\delta}{2}\right)$. Thus
	\begin{align*}
	\vol{\bigcup_{i=1}^N B\left(y_i,\frac{\delta}{2}\right)} &\leq
		\vol{B\left(0,1+\frac{\delta}{2}\right)} - \vol{B\left(0,1-\frac{\delta}{2}\right)}\\
		& = \omega_{d}\left(\left(1+\frac{\delta}{2}\right)^{d+1} - \left(1-\frac{\delta}{2}\right)^{d+1}\right)\\
		&= \omega_{d}\delta \sum_{r=0}^{d}\left(1+\frac{\delta}{2}\right)^{d-r}\left(1-\frac{\delta}{2}\right)^r \\
		& \leq \omega_{d}\delta \sum_{r=0}^{d}\left(1+\frac{\delta}{2}\right)^{d}\\
		&=\omega_{d}\delta (d+1)\left(1+\frac{\delta}{2}\right)^d \\
		& \leq \omega_d\delta(d+1)\left(\frac{3}{2}\right)^d,
	\end{align*}
	and the upper bound for $\mathcal{N}$ follows.
	
	For the lower bound, we consider the area of the spherical caps. Since all points in $S^d$ are within distance $\delta$ of the points $y_i\in\mathcal{B}$ we must have
	$$\area{S^d}=\area{\bigcup_{i=1}^{\mathcal{N}}\B{y_i}{\delta}}\leq\sum_{i=1}^{\mathcal{N}}\area{\B{y_i}{\delta}}=\mathcal{N}\area{\B{y_1}{\delta}}$$
	
	Therefore, Lemma \ref{lemma_area_spherical_cap} yields	$\displaystyle\mathcal{N}\geq \frac{\area{S^d}}{\area{\B{y_0}{\delta}}}\geq\frac{3}{d\delta^d}$.	
\end{proof}

We now proceed to prove Theorem \ref{thm_lower_bound}.

\begin{proof}[Proof of Theorem \ref{thm_lower_bound}]
	Let $G=\Bord{\epsilon(n)}{n}$. Since $\epsilon\to 0$, eventually $\epsilon<2-\lambda_d$, so by Lemma \ref{lemma_color_d+2}, $\chi(G)\leq d+2$. We now proceed to prove the lower bound by a modification of the proof of theorem \ref{thm_eps_constant}. \\
	
	Let $\delta = \sqrt{\epsilon} / 4$. Let $\mathcal{B}$ be the $\delta$-net given by Lemma \ref{lemma-nets}. Say $\mathcal{B}=\{y_1,y_2,\cdots,y_N\}$, where $$N\leq \frac{2(3^d)(d+1)}{\delta^d}=\frac{A_d}{\epsilon^{d/2}},$$ and $A_d$ is a constant which only depends on $d$. For each $i=1,\cdots, N$, define $F_i=\B{y_i}{\delta}$. Note then the $F_i's$ cover the sphere and $\diam{F_i}\leq2\delta=\frac{\sqrt{\epsilon}}{2}$. \\
	
	Applying Lemma \ref{lemma_area_spherical_cap}, we have
	$$\frac{\area{F_i}}{\area{S^d}}\geq 
	\frac{1}{4\pi}\left(\frac{\sqrt{3}}{8}\right)^{d-1}\epsilon^{d/2}=
	B_d \epsilon^{d/2}=:c,$$
	where $B_d$ is constant.\\
	
	Finally, all that remains to prove is that $1-N(1-c)^n\to 1$ as $n\to\infty$, even when $N$ and $c$ depend on $n$. This is as follows
	\begin{align*}
	N(1-c)^n &\leq \frac{A_d}{\epsilon^{d/2}}\left(1-c\right)^n \\
	&=
	\frac{A_d}{\epsilon^{d/2}}\left(1-B_d\epsilon^{d/2}\right)^n \\
	&=
	\frac{A_dn}{C^{d/2}\log{n}}\left(1-\frac{B_dC^{d/2}\log{n}}{n}\right)^n\\
	&\leq \frac{A_dn}{C^{d/2}\log{n}}\exp\left(-B_dC^{d/2}\log{n}\right)\\
	&= \frac{A_d}{C^{d/2}\log{ n}} n^{1-B_dC^{d/2}}
	\end{align*}
	The last expression goes to zero as $n\to\infty$, since $$C\geq \frac{64}{3}\sqrt[d]{\frac{3\pi^2}{4}},$$ so $B_d C^{d/2}\geq 1$ and this completes the proof. 	
\end{proof}

\begin{coro}
	If $$\frac{64}{3}\sqrt[d]{\frac{3\pi^2}{4}}\left(\dfrac{\log{n}}{n}\right)^{2/d}\leq \epsilon(n)\leq 2-\lambda_d$$ for all sufficently large $n$, then $\chi\left(\Bord{\epsilon(n)}{n}\right)=d+2$ a.a.s.
\end{coro}
\begin{proof}
	The chromatic number is monotone with respect to $\epsilon$, so this follows directly.
\end{proof}

\subsection{Upper Bound}\label{section_4}
Theorem \ref{thm_lower_bound} and its Corollary shows that if $\epsilon \to 0$ sufficiently slowly then the chromatic number of the random Borsuk graph is a.a.s. $d+2$. In this section we show that the rate obtained is tight, up to a constant factor. That is, we show an upper bound for $\epsilon$ for which the random Borsuk graph is $(d+1)$-colorable.

We start our analysis by constructing a proper coloring of $\Bor{\epsilon}\setminus\B{x}{\delta}$ with exactly $d+1$ colors, for a suitable $\delta$ that depends on $\epsilon$. Lemma \ref{lemma_area_spherical_cap} establishes that the boundary of an spherical cap on $S^d$ is a $S^{d-1}$ with radius $\delta'$, and Lemma \ref{lemma_color_d+2} allows to color it with $d+1$ colors. We will provide the technical details to translate this coloring into a proper coloring of the desired graph.\\

For the following analysis consider the spherical cap $A=\B{N}{r}$, where $N$ is the north pole. For each $x\in S^d\setminus\{N,-N\}$, let $\gamma_x:[0,\pi]\to S^d$ be the great semi-circle going from $N$ to $-N$ and passing through $x$. Define $f: S^d\to\partial A$ by letting $f(x)$ be the intersection of $\gamma_x$ with $\partial A$. Note this is a well defined function, since if $x=(x_0,\dots, x_d)$, we can parametrize $$\gamma_x(t)=\left(\frac{\sin{t}}{\sqrt{1-x_d^2}}x_0,\dots,\frac{\sin{t}}{\sqrt{1-x_d^2}}x_{d-1},\cos{t}\right)$$
so its last coordinate takes all values in $[-1,1]$ exactly once for $0\leq t\leq \pi$, and from Lemma \ref{lemma_area_spherical_cap} we know $\partial A$ consists of all points with last coordinate $a:=1-\frac{r^2}{2}$.\\

The following lemmas construct the desired coloring. 

\begin{lemma}\label{lemma_distance_geodesics}
	Let $x, y\in S^d\setminus\{N,-N\}$ such that $\norm{x-y}\leq\delta$. Define $y'=(y'_0,\dots,y'_d)$ to be the point in the geodesic $\gamma_y$, such that $y'_d=x_d$. Then $\norm{x-y'}\leq 2\delta$. 
\end{lemma}

\begin{proof}
	Without lost of generality, we may assume $y=(0,0,\dots,0,\sqrt{1-y_d^2},y_d)$, since we can get this by a rotation of the sphere that leaves the last coordinate fixed. This rotation fixes the north and south poles, so it also transforms the geodesic through $y$ into another geodesic through $y$. Thus, the formula for the geodesic simplifies to 
	$$\gamma_y=(0,\dots,0,\sin{t},\cos{t}),\text{    for    } 0\leq t\leq\pi.$$
	So, $y'=(0,\dots,0,\sqrt{1-x_d^2},x_d)$. Then 
	\begin{align*}
	\norm{x-y}^2
	&=x_0^2+\cdots+x_{d-2}^2+\left(x_{d-1}-\sqrt{1-y_d^2}\right)^2+(x_d-y_d)^2\\
	&= (1-x_d^2)+(1-y_d^2)-2x_{d-1}\sqrt{1-y_d^2}+(x_d-y_d)^2\\
	\end{align*}
	and
	\begin{align*}
	\norm{y-y'}^2
	&=\left(\sqrt{1-y_d^2}-\sqrt{1-x_d^2}\right)^2+(x_d-y_d)^2\\
	&= (1-x_d^2)+(1-y_d^2)-2\sqrt{1-x_d^2}\sqrt{1-y_d^2}+(x_d-y_d)^2
	\end{align*}

Since $x_{d-1}\leq |x_{d-1}|\leq \sqrt{x_0^2+\cdots+x_{d-1}^2}=\sqrt{1-x_d^2}$, we get $\norm{y-y'}\leq\norm{x-y}$, and so $$\norm{x-y'}\leq\norm{x-y}+\norm{y-y'}\leq 2\norm{x-y}\leq 2\delta.$$
\end{proof}

\begin{lemma}\label{lemma_distance_F}
	Let $x, y\in S^d\setminus\{N,-N\}$ such that $x\not\in A\cup (-A)$ and $\norm{x-y}\leq\delta$. Then $$\norm{f(x)-f(y)}\leq 2\delta.$$
\end{lemma}
\begin{proof}
	Let $y'\in S^d$ such that its last coordinate is $y'_d=x_d$. From the parametrization for $\gamma_x$, we see $f(x)=\gamma_x(t_1)$, where $\cos{t_1}=a$ and $\sin{t_1}=\sqrt{1-a^2}=\delta'$, the radius of $\delta A$, hence
	$$f(x)=\left(\frac{\delta'}{\sqrt{1-x_d^2}}x_0,\dots,\frac{\delta'}{\sqrt{1-x_d^2}}x_{d-1},a\right).$$
	A similar expression holds for $f(y')$, with $y'_d=x_d$, so we get
	\begin{align*}
	\norm{f(x)-f(y')}&=\sqrt{\sum_{i=0}^{d-1}\frac{\delta'^2}{1-x_d^2}(x_i-y_i')^2}\\
	&=\frac{\delta'}{\sqrt{1-x_d^2}}\sqrt{\sum_{i=0}^{d-1}(x_i-y'_i)^2}\\
	&\leq \frac{\delta'}{\sqrt{1-x_d^2}} \norm{x-y'}
	\end{align*}
	
	Moreover, since $x\not\in A\cup (-A)$, $|x_d|<a$, so $\frac{\delta'}{\sqrt{1-x_d^2}}<1$, so $\norm{f(x)-f(y')}\leq \norm{x-y'}$. Finally, if we let $y'$ be the one defined in Lemma \ref{lemma_distance_geodesics}, $f(y')=f(y)$, and therefore $\norm{f(x)-f(y)}=\norm{f(x)-f(y')}\leq \norm{x-y'}\leq 2\delta$. 
	
\end{proof}

\begin{lemma}\label{lemma_color_d+1}
	Let $0<\epsilon<1$, such that $$ r=\frac{8\sqrt{\epsilon}}{\sqrt{3(4-\lambda_{d-1}^2)}}<1,$$ $x\in S^d$, and $A=\B{x}{r}$. Let $H$ be the induced subgraph of $\Bor{\epsilon}$ by the vertex set $S^d\setminus A$. Then $\chi(H)\leq d+1$.
\end{lemma}
\begin{proof}
	Without loss of generality let $x=N$ the north pole, so $A=\B{N}{r}$. Lemma \ref{lemma_area_spherical_cap} says $\partial A$ is a $S^{d-1}$ sphere of radius $r'=r\sqrt{1-\frac{r^2}{4}}\geq\frac{\sqrt{3}}{2}r$. Thus adapting Lemma \ref{lemma_color_d+2}, we can color it in such a way that every two points with the same color are at a distance of at most $\lambda_{d-1}r'$. We then color $H$ by giving each point $y\in S^d\setminus A\setminus\{-N\}$ the color of $f(y)$, and giving the south pole $-N$ any color. We proceed to prove this is a proper coloring of $H$. \\
	
	From Lemma \ref{lemma_min_distance}, the neighbors of the south pole lie in $\B{N}{\sqrt{\epsilon-\epsilon^2/4}}\subset A$, so $-N$ is isolated in $H$. Let $y,z\in S^d\setminus A\setminus\{-N\}$ such that $\norm{y-z}>2-\epsilon$. Lemma \ref{lemma_min_distance} implies $\norm{y+z}<\delta:=2\sqrt{\epsilon-\frac{\epsilon^2}{4}}$. If we had $(-y),(-z)\in A$, that would mean $y,z\in -A$, but then $\norm{y-z}\leq r\leq 2-\epsilon$ for small $\epsilon$. So we may assume $(-y)\not\in A$, and since $y\not\in A$, $(-y)\not\in -A$. Thus $(-y)\not\in A\cup (-A)$ and $\norm{-y-z}\leq\delta$, thus Lemma \ref{lemma_distance_F} implies 
	$$\norm{f(-y)-f(z)}\leq 2\delta=4\sqrt{\epsilon-\frac{\epsilon^2}{4}}<4\sqrt{\epsilon}=\sqrt{4-\lambda_{d-1}^2}\frac{\sqrt{3}}{2}r\leq \sqrt{4-\lambda_{d-1}^2}r'$$
	From the definition of $f$, it is clear that $f(-y)=-f(y)$, thus Lemma \ref{lemma_min_distance} implies $\norm{f(y)-f(z)}>\lambda_{d-1}r'$, and so $f(y)$ and $f(z)$ have different colors, meaning $y$ and $z$ have different colors as well. Therefore $\chi(H)\leq d+1$. 
		
	\end{proof}

As an immediate application, if a random Borsuk graph leaves some spherical cap in $S^d$ of radius bigger than $r$ with no vertices, then it can be colored with $d+1$ colors. We will show that this is indeed the case when $\epsilon\to 0$ at the said rate. We now include some theorems about Poisson Point Processes and Poisson distributions. For their proofs and a complete discussion refer to \cite{Penrose2003} or \cite{Kingman1993}. 

\begin{theorem}[Poissonization]\label{thm_poissonization}
	Let $X_1, X_2, \dots$, be uniform random variables on $S^d$. Let $M\sim \text{Pois}(\lambda)$ and let $\eta$ be the random counting measure associated to the point process $P_\lambda$= $\{X_1,X_2,\dots,X_M\}$. Then $P_\lambda$ is a Poisson Point Process and for a Borel $A\subset S^d$, $\eta(A)\sim \pois{\lambda\frac{\area{A}}{\area{S^d}}}$.
\end{theorem}

\begin{lemma}\label{lemma_poisson}
	For $n\geq 0$, $\P{\pois{2n}<n}\leq e^{-0.306n}.$
\end{lemma}

We are now ready to prove the Theorem \ref{thm_upper_bound}.

\begin{proof}[Proof of Theorem \ref{thm_upper_bound}]
	Let $X_1, X_2, \dots$, be uniform random variables on $S^d$. Let $M\sim\pois{2n}$. Let $\eta$ be the random counting measure of the Poisson Point Process $\left\{X_1,\dots,X_M\right\}$. Similarly, let $\eta_1^n$ be the counting measure of the Random points $\left\{X_1,\dots, X_n\right\}$.\\
	
    Let $$\delta=\frac{16\sqrt{\epsilon}}{\sqrt{3(4-\lambda_{d-1}^2)}}=A_d\sqrt{\epsilon},$$ where $A_d$ is a constant which only depends on $d$. Let $\mathcal{B}=\left\{y_1,\dots,y_N\right\}$ be the $\delta$-net given by Lemma \ref{lemma-nets}, so $$N \geq \frac{3}{d\delta^d}=\frac{B_d}{\epsilon^{d/2}},$$ and $B_d$ is constant. Let $F_i=\B{y_i}{\delta/2}$ for $i=1,\dots,N$ be spherical caps centered at the $\delta$-net. Thus, as in the proof of \ref{lemma-nets}, the $F_i$'s are disjoint.
	Lemma \ref{lemma_area_spherical_cap} gives $$\frac{\area{F_i}}{\area{S^d}}\leq\frac{d}{3}\left(\frac{\delta}{2}\right)^d=D_d\epsilon^{d/2},$$
where $D_d$ is constant.
	
	Note that these spherical caps have the same radius required by Lemma \ref{lemma_color_d+1}, so if we prove that a.a.s.\ one of these $F_i$'s doesn't contain any vertices of the random Borsuk graph, then it must be contained in $S^d\setminus F_i$, and the Lemma \ref{lemma_color_d+1} gives a proper $(d+1)$ coloring. This is what we do.\\
	
	Note that
	\begin{equation}
		\P{\min_{1\leq i\leq N}\eta(F_i)=0} \leq \P{\min_{1\leq i\leq N}\eta_1^n(F_i)=0}+\P{M<n}.
		\label{eqn_prob_unif_vs_poisson}
	\end{equation} 
	
We have
	\begin{align*}
	\P{\min_{1\leq i\leq N}\eta(F_i)=0}&=1-\P{\bigwedge_{i=1}^N\eta(F_i)>0}=1-\prod_{i=1}^N\P{\eta(F_i)>0}\\&=1-\left(1-\P{\pois{2n\frac{\area{F_1}}{\alpha_d}}=0}\right)^N\\
	&\geq 1-\exp\left(-\exp\left(-2n\frac{\area{F_1}}{\alpha_d}\right)N\right)\\
	&\geq 1-\exp\left(-\exp(-2nD_d\epsilon^{d/2})\frac{B_d}{\epsilon^{d/2}}\right)\\
	&=1-\exp\left(-\frac{B_d}{C^{d/2}\log{n}}n^{1-2D_dC^{d/2}}\right).\\
	\end{align*}

    This last expression tends to 1 as $n\to\infty$, since $C$ is such that $1-2D_dC^{d/2}>0$.\\

Lemma \ref{lemma_poisson} assures that $$\P{M<n}=\P{\pois{2n}<n}\to 0$$ as $n\to \infty,$ and therefore (\ref{eqn_prob_unif_vs_poisson}) gives $\displaystyle\P{\min_{1\leq i\leq N}\eta_1^n(F_i)=0}\to 1$, as desired.\\
\end{proof}

\section{Further Questions}

\begin{enumerate}
    
\item It might be possible to find sharper constants in Theorems \ref{thm_lower_bound} and \ref{thm_upper_bound}. For $d=1$, it is certainly possible. The following can be achieved with similar methods to the ones used throughout this paper, so we include the statement without proof.

\begin{theorem}
Let $\epsilon=C \left(\log n / n \right)^2$.
\begin{enumerate}
    \item If $C\geq 9\pi^2 /4$, then a.a.s. $\chi\left(\Borone{\epsilon}{n}\right)=3$.
    \item If $C<\pi^2/4$, then a.a.s. $\chi \left( \Borone{\epsilon}{n} \right) \le 2$.
    \end{enumerate}
\end{theorem}

\item We wonder whether there exist functions $\epsilon=\epsilon(n)$ such that the chromatic number of the random Borsuk graph $\Bord{\epsilon}{n}$ a.a.s.\ equals $i$, for $1 \le i \le d+1$. 

\item We only studied here the case that $d$ is fixed and $\epsilon$ is either fixed or tends to zero at some rate. It also seems interesting to let $d \to \infty$ at some rate, or to let $d$ be fixed and $\epsilon \to 2$. See, for example, Raigorodskii's work on coloring high-dimensional spheres \cite{Raigorodskii2012}. 
\end{enumerate}

\noindent We thank our anonymous referees for careful reading and helpful comments.

\bibliographystyle{plain}
\bibliography{References}

\end{document}